\documentclass[10pt]{amsart}
\usepackage{enumerate,amsmath,amssymb,latexsym,
amsfonts, amsthm, amscd, MnSymbol}

\theoremstyle{plain}

\newtheorem{theorem}{Theorem}

\numberwithin{equation}{section}

\addtocounter{section}{-1}

\begin{document}

\title {The Peirce axiom scheme and suprema}

\date{}

\author[P.L. Robinson]{P.L. Robinson}

\address{Department of Mathematics \\ University of Florida \\ Gainesville FL 32611  USA }

\email[]{paulr@ufl.edu}

 \keywords{}

\begin{abstract}

We present equivalents to the Peirce axiom scheme and highlight its relevance for pairwise suprema.

\end{abstract}

\maketitle

\medbreak

With the conditional ($\supset$) as its only logical connective and with modus ponens (MP) as its only inference rule, the Implicational Propositional Calculus (IPC) is commonly founded on the following three axiom schemes: \\ 

($IPC_0$) $\; \; [(A \supset B) \supset A] \supset A\; \; $ ({\it the Peirce axiom scheme}); \\

($IPC_1$) $\; \; A \supset (B \supset A)\; \; $ ({\it affirmation of the consequent}); \\

($IPC_2$) $\; \; [A \supset (B \supset C)] \supset [(A \supset B) \supset (A \supset C)]\; \; $ ({\it self-distributivity}). \\

Our aim here is to offer equivalents for the Peirce axiom scheme and highlight a significant r\^ole that it plays within the IPC. A convenient DIY introduction to the IPC may be found on pages 24-25 of [4]. 

\medbreak 

Throughout, we shall assume modus ponens and the axiom schemes $IPC_1$ and $IPC_2$. This comment applies in particular to the statement of theorems: MP, $IPC_1$ and $IPC_2$ will be implicit assumptions in each of them; whether or not the Peirce scheme $IPC_0$ is assumed will be stated explicitly. It is in this sense that our equivalents to the Peirce scheme should be understood: each of them is equivalent to $IPC_0$ in the presence of MP, $IPC_1$ and $IPC_2$. 

\medbreak 

It will be convenient to begin by recalling some elementary properties of IPC that do not require $IPC_0$. Perhaps the most useful of these is the Deduction Theorem (DT):  if $\Gamma, A \vdash B$ (that is, if there is a deduction of $B$ from $A$ and the set $\Gamma$ of well-formed formulas) then $\Gamma \vdash A \supset B$ (that is, there is a deduction of $A \supset B$ from $\Gamma$ alone); as a special case, if $A \vdash B$ then $A \supset B$ is a theorem. Because of this, we shall feel free to write $A \vdash B$ and $\vdash A \supset B$ interchangeably without mention. A particular consequence of DT is Hypothetical Syllogism (HS): $A \supset B, B \supset C \vdash A \supset C$ as a derived inference rule; alternatively, if $A \vdash B$ and $B \vdash C$ then $A \vdash C$. Another useful elementary fact is that $A \supset A$ is a theorem of IPC without $IPC_0$; it typically appears as a step on the way to DT (from which metatheorem it otherwise follows at once). 

\medbreak 

One direction of the following equivalence already appears in [4]. 

\medbreak 

\begin{theorem} \label{IPC_0'}
The Peirce axiom scheme $IPC_0$ is equivalent to the axiom scheme \par
${\rm (} IPC_0' {\rm )}$ $\; \; \; \; (A \supset Q) \supset \{ [ (A \supset B) \supset Q] \supset Q \}.$ 
\end{theorem} 

\begin{proof} 
Assume $IPC_0$. HS yields 
$$A \supset Q, Q \supset B \vdash A \supset B$$
whence DT yields 
$$A \supset Q \vdash (Q \supset B) \supset (A \supset B).$$
Consequently, from $A \supset Q$ and $(A \supset B) \supset Q$ as hypotheses, we deduce $(Q \supset B) \supset Q$ by HS. Now, $IPC_0$ gives $\vdash [(Q \supset B) \supset Q] \supset Q$ whence an application of MP produces 
$$A \supset Q, (A \supset B) \supset Q \vdash Q$$
and two applications of DT bring us to $IPC_0'$ as a theorem scheme. \par 
Assume $IPC_0'$ as stated. With $Q$ replaced by $A$ this yields 
$$ \vdash (A \supset A) \supset \{ [ (A \supset B) \supset A] \supset A \}.$$ 
As $A \supset A$ is a theorem, an application of MP brings us to $IPC_0$ as a theorem scheme. 
\end{proof} 

\medbreak 

The approach to our second equivalent for the Peirce axiom scheme starts from a simple observation within the classical Propositional Calculus. As regards its truth-functional nature, disjunction may be expressed in terms of the conditional: explicitly, all Boolean valuations agree on $A \vee B$ and $(A \supset B) \supset B$. This is one of the less obvious semantic relationships between logical connectives, among the many to which Smullyan draws attention in [5] and elsewhere. On this point, see also D4 in Section 11 of [1]. 

\medbreak 

Accordingly, let us define (as an abbreviation) 
$$A \vee B : = (A \supset B) \supset B.$$

\medbreak 

The following is called $\vee$-Introduction in [3]. 

\medbreak

\begin{theorem} \label{Introduction}
$A \vdash A \vee B$ and $B \vdash A \vee B.$
\end{theorem} 

\begin{proof} 
For the first deduction, MP gives $A, A \supset B \vdash B$ and then DT gives $A \vdash (A \supset B) \supset B$. The second amounts to an instance of $IPC_1$. 
\end{proof} 

\medbreak 

The result of the following theorem is called $\vee$-Elimination in [3]. 

\medbreak 

\begin{theorem} \label{Elimination}
Assume the Peirce axiom scheme. If $A \vdash Q$ and $B \vdash Q$ then $A \vee B \vdash Q.$ 
\end{theorem} 

\begin{proof} 
The equivalent to the Peirce scheme presented in Theorem \ref{IPC_0'} justifies 
$$A \supset Q \vdash [(A \supset B) \supset Q] \supset Q$$
while an application of HS justifies 
$$B \supset Q, (A \supset B) \supset B  \vdash (A \supset B) \supset Q.$$
An application of MP now justifies 
$$A \supset Q, B \supset Q, (A \supset B) \supset B \vdash Q$$
whence an application of DT justifies 
$$A \supset Q, B \supset Q \vdash [(A \supset B) \supset B] \supset Q$$
or 
$$A \supset Q, B \supset Q \vdash (A \vee B) \supset Q.$$
Finally, if $A \vdash Q$ and $B \vdash Q$ then $\vdash A \supset Q$ and $\vdash B \supset Q$ so that $\vdash (A \vee B) \supset Q$ and therefore $A \vee B \vdash Q$. 
\end{proof} 

\medbreak 

We may regard this theorem as providing formal justification, on the basis of the Peirce axiom scheme, for the inference rule `from $A \vdash Q$ and $B \vdash Q$ infer $A \vee B \vdash Q$' or 
$$ \frac{A \supset Q, B \supset Q}{(A \vee B) \supset Q}$$
which we shall refer to as $\vee$E (for $\vee$-Elimination). 

\medbreak 

\medbreak 

We have thus established one direction in the following result, which presents a new equivalent to the Peirce axiom scheme; recall our standing assumption that MP, $IPC_1$ and $IPC_2$ are in force. 

\medbreak 

\begin{theorem} \label{veeE}
The Peirce axiom scheme $IPC_0$: 
$$ [(A \supset B) \supset A] \supset A$$
is equivalent to the inference rule $\; \vee$E:  
$$A \supset Q, B \supset Q \vdash (A \vee B) \supset Q.$$
\end{theorem} 

\begin{proof} 
For the reverse direction, assume the inference rule $\vee$E. Now, hypothesize $(A \supset B) \supset A$. This hypothesis has as a trivial consequence the theorem $A \supset A$. An application of HS to our hypothesis and the instance $B \supset (A \supset B)$ of $IPC_1$ yields the consequence $B \supset A$. Two applications of MP give 
$$(A \supset B) \supset A, A \supset B \vdash B$$  
whence DT reveals that our hypothesis also has $(A \supset B) \supset B \; (= A \vee B$) as a consequence. In summary, $(A \supset B) \supset A$ has as consequences $A \supset A$, $B \supset A$, $A \vee B$. Next, our new inference rule $\vee$E with $Q : = A$ provides the further deduction 
$$A \supset A, B \supset A \vdash (A \vee B) \supset A.$$
Lastly, MP applied to the consequences $A \vee B$ and $ (A \vee B) \supset A$ results in 
$$(A \supset B) \supset A \vdash A$$ 
and by DT we arrive at $[(A \supset B) \supset A] \supset A$ as a theorem scheme. 
\end{proof} 

\medbreak 

A reformulation of this new equivalent to the Peirce axiom scheme is of interest. 

\medbreak 

Let $wf$ denote the set comprising all well-formed formulas. Initially, assume only MP along with $IPC_1$ and $IPC_2$. Declare $A \in wf$ and $B \in wf$ to be (syntactically) {\it equivalent} precisely when both $A \vdash B$ and $B \vdash A$ (thus, precisely when both $A \supset B$ and $B \supset A$ are theorems). This evidently defines an equivalence relation $\equiv$ on $wf$; transitivity holds by virtue of HS. Write $\mathbb{L}$ for the set of $\equiv$-classes in $wf$ and write $[A]$ for the $\equiv$-class of $A$. 

\medbreak 

The set $\mathbb{L}$ is partially ordered by declaring that $[A] \leqslant [B]$ precisely when $A \vdash B$ (equivalently, precisely when $A \supset B$ is a theorem); it is readily checked that this does indeed well-define a partial order. Notice that the poset $(\mathbb{L}, \leqslant)$ has a top element: namely, $[T]$ where $T$ is any theorem; if $A$ is any well-formed formula then $A \vdash T$ is a triviality. By contrast, $(\mathbb{L}, \leqslant)$ has no bottom element; semantically, a well-formed formula representing a bottom element must take the value 0 in any Boolean valuation, but each well-formed formula built from $\supset$ alone takes the value 1 when every propositional variable is assigned 1 as its value. 

\medbreak 

In these terms, we now have the following reformulation of Theorem \ref{veeE} (with our usual standing assumption). 

\begin{theorem} \label{sup} 
The Peirce axiom scheme $IPC_0$ is equivalent to the requirement that each pair $[A], [B]$ in $\mathbb{L}$ has $[A \vee B]$ as its supremum. 
\end{theorem} 

\begin{proof} 
Theorem \ref{Introduction} tells us that $[A]$ and $[B]$ have $[A \vee B]$ as an upper bound. Theorem \ref{Elimination} tells us that if the Peirce axiom scheme holds then $[A \vee B]$ is actually the least upper bound of $[A]$ and $[B]$. The reverse direction of Theorem \ref{veeE} implies that if $[A \vee B]$ is the least upper bound of $[A]$ and $[B]$ then $IPC_0$ holds. This concludes the proof. 
\end{proof} 

\medbreak 

Thus, for the Implicational Propositional Calculus, $(\mathbb{L}, \leqslant)$ is a (topped, bottomless) join-semilattice in which pairwise suprema are given by 
$$\sup \{ [A], [B] \} = [(A \supset B) \supset B].$$
For an introductory account of semilattices, we refer to [2].

\bigbreak

\begin{center} 
{\small R}{\footnotesize EFERENCES}
\end{center} 
\medbreak 

[1]  A. Church, {\it Introduction to Mathematical Logic}, Princeton University Press (1956). 

[2] G. Gr\"atzer, {\it Lattice Theory - First Concepts and Distributive Lattices}, W.H. Freeman (1971); Dover Publications (2009). 

[3] S.C. Kleene, {\it Mathematical Logic}, John Wiley (1967); Dover Publications (2002). 

[4] J. W. Robbin, {\it Mathematical Logic - A First Course}, W.A. Benjamin (1969); Dover Publications (2006).

[5] R.M. Smullyan, {\it Logical Labyrinths}, A.K. Peters (2009). 

\medbreak

\end{document}